\documentclass[a4paper, oneside, 11pt]{article}
\usepackage[english]{babel}
\usepackage{latexsym,amsfonts,amsmath,enumerate,graphics,enumerate,amsthm}
\title{\textbf{Bands in partially ordered vector spaces with order unit} }
\author{Anke Kalauch, Bas Lemmens and Onno van Gaans}
\date{}

\let\epsilon=\varepsilon
\newtheorem{theorem}{Theorem}[]
\newtheorem{lemma}[theorem]{Lemma}
\newtheorem{corollary}[theorem]{Corollary}

\newtheorem{proposition}[theorem]{Proposition}

\newtheorem{question}[theorem]{Question}
\theoremstyle{definition}
\newtheorem{definition}[theorem]{Definition}
\newtheorem{example}[theorem]{Example}
\newtheorem{remark}[theorem]{Remark}
\newtheorem*{acknowledgement}{Acknowledgement}

\let\epsilon=\varepsilon
\let\phi=\varphi
\let\theta=\vartheta
\newcommand{\R}{\mathbb{R}}
\newcommand{\N}{\mathbb{N}}
\newcommand{\RR}{\mathbb{R}}

\newcommand{\car}{\operatorname{car}}
\newcommand{\Lbar}{\overline{\Lambda}}

\newcommand{\spn}{\mathrm{span}\,}
\newcommand{\sat}{\mathrm{sat}\,}
\newcommand{\aff}{\mathrm{aff}\,}

\begin{document}
\maketitle
\begin{abstract}
In an Archimedean directed partially ordered vector space $X$, one can define the concept of a band in terms of disjointness. Bands can be studied by using a vector lattice cover $Y$ of $X$. If $X$ has an order unit, $Y$ can be represented as $C(\Omega)$, where $\Omega$ is a compact Hausdorff space.  
We characterize bands in $X$, and their disjoint complements, in terms of subsets of $\Omega$. We also  analyze  two methods to extend bands in $X$ to $C(\Omega)$  and show how the carriers of a band and its extensions are related.

We use the results to show that in each $n$-dimensional partially ordered vector space with a closed generating cone, the number of bands is bounded by $\frac{1}{4}2^{2^n}$ for $n\geq 2$.  We also construct examples of $(n+1)$-dimensional partially ordered vector spaces with ${2n\choose n}+2$ bands. This shows that there are $n$-dimensional partially ordered vector spaces that have more bands than an $n$-dimensional Archimedean vector lattice when $n\geq 4$. 
\end{abstract}

\section{Introduction}\label{sec1}

Bands play an important role in the theory of vector lattices.
Together with concepts such as ideals and disjointness they are 
widely used to study the structure of vector lattices.
Bands and disjointness are closely related, as in Archimedean vector lattices bands are exactly those subspaces that coincide with their double disjoint complements. 
In the study of operators on vector lattices there is a particular interest in operators preserving certain structures of the lattice. Especially the study of band preserving and disjointness preserving operators is an active field of research, see e.g. \cite{abrkit2000,abrkit2002,bou2008,randrian2005}.

If one considers vector lattices, one naturally encounters ordered spaces that are not lattices. For instance, linear subspaces of vector lattices
and spaces of operators between vector lattices are partially ordered vector spaces, but not lattices, in general. It is therefore of interest to extend the theory of vector lattices and operators on them
to partially ordered vector spaces.  

The notion of a band is defined in partially ordered vector spaces by means of disjointness in \cite{GaKa2006}. Fundamental properties of bands in partially ordered vector spaces are shown in \cite{gaanskalauch} with the aid of vector lattice covers. In \cite{KalLemGaa2010} a functional representation method provides a vector lattice cover of an Archimedean partially ordered vector space $X$ with unit by embedding $X$ into a space of continuous functions $C(\Omega)$, where $\Omega$ is a compact Hausdorff space. This raises the question whether the bands in $X$ can be identified in terms of subsets of $\Omega$. 
In Section \ref{sec2} we establish a one-to-one correspondence between the bands in $X$ and so-called `bisaturated' subsets of $\Omega$. 

It has been shown in \cite{gaanskalauch} that bands in $X$ can be extended to bands in a vector lattice cover of $X$. In Section \ref{sec3} we consider two different extension methods using disjoint complements, for which we establish the relations between the `carrier' of a band in $X$ and the carrier of the extension of the band in $C(\Omega)$.

In Section \ref{sec4} we use the results of Section \ref{sec2} to investigate bands in finite dimensional spaces. We show that the results  provide a method to list all bands in spaces with polyhedral cones. Further, we prove that the number of bands in a space with a polyhedral cone is bounded by a constant only depending on the dimension of the space. By  a plastering approximation we extend this result to arbitrary generating closed cones. 

In the remainder of this section we fix our notation.


Let $X$ be a real vector space and let $K$ be a \emph{cone} in $X$, i.e., $K$ is a \emph{wedge} ($x,y\in K, \,\lambda,\mu\geq 0$ imply $\lambda
x+\mu y\in K$) and $K\cap (-K)=\{0\}$. In $X$ a partial order is
introduced by defining $y\ge x$ if and only if $y-x\in K$. In this case $X$ is called a \emph{partially ordered vector space}. The space $X$ is called \emph{directed} if $X=K-K$; in this case $K$ is called \emph{generating}.
The space $X$ is \emph{Archimedean} if for any $x,y\in X$ with $nx\leq y$ for all $n\in \mathbb{N}$ we have $x\leq 0$. 

For a subset $M$ of $X$ the set of all upper 
bounds of $M$ is denoted by $M^u$.
As in \cite[Definition 2.1]{gaans2003}, we call $M\subseteq X$  \emph{solid} if from $x\in X$, $y\in M$ and $\{x,-x\}^u\supseteq \{y,-y\}^u$ it follows that $x\in M$. A solid linear subspace of $X$ is called an \textit{ideal}. 
According to \cite{GaKa2006}, $x,y\in X$ are called \emph{disjoint}, in symbols $x\perp y$, if \[\{x+y,x-y\}^u=\{x-y,-x+y\}^u.\] 
The \emph{disjoint complement} of a subset $M$ of $X$ is the set \[M^d:=\{y\in X\colon x\perp y \mbox{ for all } x\in M\}\,.\] 
A linear subspace $B$ of $X$ is called a \emph{band} if $(B^d)^d=B$. 
If $X$ is an Archimedean vector lattice, these notions coincide with the usual ones. In a directed Archimedean partially ordered vector space $X$ 
every band is an ideal \cite[Theorem 5.14]{gaanskalauch}, 
and for a subset $B\subseteq X$ the set $B^d$ is a band \cite[Proposition 5.5 (ii)]{GaKa2006}. 
 
Throughout the exposition we assume that $X$ is an Archimedean partially ordered vector space and has an \emph{order unit} $u\in K$, 
i.e., for each $x\in X$ there exists $\alpha\geq 0 $ such that
\[
-\alpha u\leq x\leq \alpha u.
\]
Then $X$ is directed.
As usual, an order unit $u\in K$ gives rise to a norm $\left\|\cdot\right\|_u$ on $X$ by 
\[
\|x\|_u=\inf\{\alpha>0\colon -\alpha u\leq x\leq \alpha u\}.
\]
 With respect to this norm each 
positive linear functional $\phi\colon X\to\mathbb{R}$ is continuous, as 
$|\phi(x)|\leq \phi(u)$ for all $x\in X$ with $\|x\|_u\leq 1$.
In the norm dual $X'$ the set \[K':=\{\varphi\in X'\colon \varphi(K)\subseteq [0,\infty)\}\] is a cone.

Next we recall the functional representation of $X$. Denote \[\Sigma:=\{\phi\in K'\colon \phi(u)=1\}.\] 
 By the Banach-Alaoglu theorem  the closed unit ball $B'$ of $X'$ 
is w*-compact. $\Sigma$ is a w*-closed subset of $B'$, and hence w*-compact. 
 Let 
 \begin{equation*}
 \Lambda:=\{\phi\in\Sigma \colon \phi\mbox{ is an extreme point of }\Sigma\}.
 \end{equation*}
  Recall that $\phi \in \Sigma$ is \emph{extreme} if $\phi$ is not in the (relative) interior of a line-segment in $\Sigma$, and that $\Lambda\neq  \emptyset$ by the Krein-Milman theorem. In general, $\Lambda$ need not be w*-closed, not even if $X$ is finite dimensional.   
For $M\subseteq X'$ the w*-closure of $M$ 
is denoted by 
$\overline{M}$. 
Thus $\overline{\Lambda}$ is a compact Hausdorff space. 
 
In the sequel we write $\Phi\colon X\to C(\overline{\Lambda})$ to denote the linear map 
\begin{equation}\label{eq.greeks2}
(\Phi(x))(\phi)=\phi(x)\mbox{\quad for }\phi\in\overline{\Lambda}.
\end{equation}
$(\Phi,\overline{\Lambda})$ is a functional representation of $X$, which means that $\Phi$ is linear, bipositive, maps $u$ to the constant 1 function (which implies isometry with respect to the $u$-norm) and the image of $\Phi$ separates the points of $\overline{\Lambda}$. 
Recall that a linear map $\Psi\colon U\to V$, where $U$ and $V$ are partially ordered vector spaces, is called \emph{bipositive} if for every $x\in U$ one has 
$x\ge 0$ if and only if $\Psi(x)\ge 0$.

We proceed by recalling the vector lattice cover. 
A linear subspace $D$ of a partially ordered vector space $Y$ is \emph{order dense} if every 
$h\in Y$ is the greatest lower bound of the set $\{g\in D\colon\, g\ge h\}$. 
A \emph{vector lattice cover} of a partially ordered vector space $X$ is a vector lattice $Y$ for which there exists a bipositive linear map $i\colon X\to Y$ such that $i(X)$ is order dense in $Y$. If in addition there is no proper vector sublattice of $Y$ containing $i(X)$, then $(Y,i)$ is called a \emph{Riesz completion} \cite{Ha1993}.  
A Riesz completion $Y$ is unique up to isomorphism \cite[Remark 3.2]{Ha1993}, and hence we may speak of \emph{the Riesz completion} of $X$.
The partially ordered vector spaces that have a Riesz completion are called \emph{pre-Riesz} spaces.
It is shown in \cite[Theorems 1.7, 3.5, 3.7, 4.13 and Corollaries 4.9--4.11]{Ha1993} that, among other spaces, every directed Archimedean partially ordered vector space $X$ has a Riesz completion. If, in addition, there is an order unit in $X$, the Riesz completion can be obtained by means of $\Phi$.   
\begin{theorem}\emph{\cite[Theorem 10]{KalLemGaa2010}} \label{thm:main}
Let $X$ be an Archimedean partially ordered vector space with order unit $u$. Then $C(\overline{\Lambda})$ is a vector lattice cover of $X$. Moreover, the Riesz subspace of $C(\overline{\Lambda})$ generated by $\Phi(X)$ is the Riesz completion of $X$.
\end{theorem} 

If $(Y,i)$ is a vector lattice cover of $X$, then $x\perp y$ if and only if $i(x)\perp i(y)$ \cite[Theorem 4.1]{GaKa2006}. In this setting, the extension property for bands is shown in \cite[Proposition 5.12]{gaanskalauch}, namely, if $B$ is a band in $X$, then there is a band $\widehat{B}$ in $Y$ such that $i(B)=\widehat{B}\cap i(X)$. 
To describe $B$ by means of the functional representation, we first
recall well-known results concerning ideals
and bands in $C(\Omega)$, where $\Omega$ is a compact Hausdorff space,
cf. e.g. \cite[Example 9.4]{Zaa1997}.
For a set $S\subseteq C(\Omega)$ the \emph{carrier} of $S$ is defined by \[\operatorname{car}(S)=\left\{\omega\in \Omega\colon \mbox{ there is } s\in S \mbox{ such that } s(\omega)\neq 0\right\}.\]  
The carrier of an ideal in $C(\Omega)$ is an open subset of $\Omega$.
Different ideals in $C(\Omega)$ may have the same carrier.
For an open subset $O$ of $\Omega$ the set $I_O:=\{s\in C(\Omega)\colon \forall \omega\in\Omega\setminus O \mbox{ one has } s(\omega)=0 \}$ is the largest ideal having $O$ as its carrier.
If $P$ is an open subset of $\Omega$ as well, then $I_O=I_P$ if and only if $O=P$. 

Recall that $f,g\in C(\Omega)$ are disjoint if and only if for every $\omega\in\Omega$ one has $f(\omega)=0$ or $g(\omega)=0$.
For every band $B$ in $C(\Omega)$ there is an open subset $O$ of $\Omega$ such that $B=I_O$, so a band is determined whenever its carrier is known.
On the other hand,
for an open subset $O$ of $\Omega$ the set $I_O$ is not a band, in general. An open subset $O$ of $\Omega$ is called \textit{regularly open} if it equals the interior of its closure, i.e. $O=\operatorname{int}\left(\overline{O}\right)$. 
Bands in $C(\Omega)$ are characterized as follows (see for instance \cite[12.7--12.9]{JonRoo1977} or \cite[Example 9.4]{Zaa1997}). 
\begin{proposition} \label{pro2}
For an open subset $O$ of $\Omega$ the ideal $I_O$ is a band if and only if $O$ is regularly open. 
\end{proposition}
We combine properties of disjointness in vector lattice covers and in $C(\Omega)$. If $X$ is an Archimedean partially ordered vector space with unit and $\Phi$ is its embedding into $C(\overline{\Lambda})$ as given above, then Theorem \ref{thm:main} yields that for $x,y\in X$ 
\begin{eqnarray*}
x\perp y \mbox{ in } X&\Longleftrightarrow & \Phi(x)\perp \Phi(y)\mbox{ in } C(\overline{\Lambda})\\ &\Longleftrightarrow & \mbox{for each } \varphi\in\overline{\Lambda} \mbox{ one has } \varphi(x)=0 \mbox{ or } \varphi(y)=0.
\end{eqnarray*}
Moreover, if $B$ is a band in $X$, then there exists a regularly open set $O\subseteq \overline{\Lambda}$ such that 
\[B=\{x\in X\colon    \varphi(x)=0 \mbox{ for all }\varphi\in\overline{\Lambda}\setminus O\}.\]

\section{Characterization of bands}\label{sec2}

We consider an Archimedean partially ordered vector space $X$ with order unit and its embedding $\Phi$ into $C(\Lbar)$ as given in Section \ref{sec1}. For $B\subseteq X$ we will deal with $\operatorname{car}(\Phi(B))$, which we abbreviate by $\operatorname{car}(B)$.

\begin{remark}
If $X=C(\Omega)$, where $\Omega$ is a compact Hausdorff space, then $\Lambda=\overline{\Lambda}$ and $\Omega$ are homeomorphic, such that the carrier of $B\subseteq C(\Omega)$ defined as a subset of $\overline{\Lambda}$ corresponds to the carrier of $B$ seen as a subset of $\Omega$. In that sense the definition of the carrier in $X$ is compatible with the definition of the carrier in Section 1.
\end{remark}
We denote 
\[N(B):=\overline{\Lambda}\setminus \car(B)=\{\varphi\in\overline{\Lambda}\colon\, \varphi(b)=0\mbox{ for all }b\in B\}.\]
For a set $M\subseteq \Lbar$ we also use the notation $M^c:=\overline{\Lambda}\setminus M$. $\spn(M)$ denotes the linear subspace of $X'$ spanned by $M$ and the affine hull of $M$
is given by \[\operatorname{aff}(M)=\left\{\sum_{i=1}^n\alpha_i m_i\colon n\in\N,\, \alpha_i\in\R,\, \sum_{i=1}^n\alpha_i=1,\, m_i\in M \right\}.\]
For $M\subseteq \Lbar$ denote the \emph{zero set induced by $M$} by \[Z(M)=\{x\in X\colon \varphi(x)=0\mbox{ for all } \varphi\in M \}.\] 
In the present section, we will characterize those sets $M\subseteq \Lbar$
that belong to a band $B$ in $X$ in the sense that $B=Z(M)$ and $M=N(B)$.
We need the following notion.

\begin{definition} 
For $M\subseteq \Lbar$ we define the \emph{saturation} of $M$ by 
\[\sat (M):=\Lbar\cap\overline{\spn (M)}.\]
A set $M\subseteq \Lbar$ is called \emph{saturated} if $M=\sat (M)$.
\end{definition}
Instead of the linear span  one could also use the affine hull in the definition of saturation, see Corollary \ref{cor6} below.

It is straightforward that for $B\subseteq X$ the set $N(B)$ is saturated.  In the subsequent proposition we characterize the saturation by means of zero sets. 
Let us first state two properties.
\begin{lemma}\label{twobasicprops}
Let $M\subseteq \overline{\Lambda}$. Then the following holds:
\begin{itemize}
\item[(i)] $\sat(\sat (M))=\sat (M)$;
\item[(ii)] $Z(\sat (M))=Z(M)$.
\end{itemize}
\end{lemma}
\begin{proof}
(i) Clearly, $M\subseteq \sat (M)$ and if $N\subseteq\overline{\Lambda}$ is such that $M\subseteq N$ then $\sat (M)\subseteq\sat (N)$. Since $\sat (M)\subseteq \overline{\spn (M)}$ it follows that 
\[\overline{\spn (\sat (M))}\subseteq \overline{\spn (M)},\] 
so that $\sat(\sat (M))\subseteq \sat (M)$. Hence $\sat(\sat (M))=\sat (M)$. 

(ii) $\sat (M)\supseteq M$, so $Z(\sat (M))\subseteq Z(M)$. If $x\in Z(M)$, then $\varphi(x)=0$ for all $\varphi\in M$, so $\varphi(x)=0$ for all $\varphi\in \overline{\spn (M)}$. Hence $x\in Z(\sat (M))$.
\end{proof}

The definitions of $N(B)$ and $Z(M)$, where $B\subseteq X$ and $M\subseteq\overline{\Lambda}$, are closely related to the notion of annihilator. Indeed, $N(B)=\overline{\Lambda}\cap B^\perp$ and $Z(M)={}^\perp M$, where $B^\perp$ denotes the annihilator of $B$ in $X'$ and ${}^\perp M$ the annihilator of $M$ in $X$ (see \cite[Definition 1.10.14]{megginson}). The next proposition is a variant of \cite[Proposition 2.6.6(b)]{megginson}.

\begin{proposition} \label{charsatset}
Let $M\subseteq \Lbar$. Then
\[\sat (M)=N(Z(M)).
\]
\end{proposition}
\begin{proof}
If $\varphi\in \Lbar\cap\spn (M)$ and $x\in Z(M)$, then $\varphi(x)=0$,
so 
\begin{eqnarray*}
\Lbar\cap\spn (M)&\subseteq&\{\varphi\in\Lbar\colon \varphi(x)=0\mbox{ for all } x\in Z(M) \}\\
&=&\bigcap_{x\in Z(M)}\{\varphi\in\Lbar\colon \varphi(x)=0\}.
\end{eqnarray*}
Since the right hand side of the equality is a w*-closed set, it follows that 
\[\Lbar\cap\overline{\spn (M)} \subseteq \{\varphi\in\Lbar\colon \varphi(x)=0\mbox{ for all } x\in Z(M) \}.\]
Vice versa, let 
\[\psi\in \Lbar\setminus \overline{\spn (M)}.\]
By the Hahn-Banach theorem, 
there is $x\in X$ such that $\psi(x)\neq 0$ and for all $\varphi\in \overline{\spn (M)}$ one has $\varphi(x)=0$.
In particular, $x\in Z(M)$.
So, 
\[\psi\not\in \{\varphi\in\Lbar\colon \varphi(x)=0\mbox{ for all } x\in Z(M) \}.\]
\end{proof}
It can be observed that the saturation of $M\subseteq\overline{\Lambda}$ is the largest subset $N$ of $\overline{\Lambda}$ such that $Z(N)=Z(M)$. 

In the considerations below it appears more natural to use the affine hull instead of the linear span in describing saturations. It is a consequence of Proposition \ref{charsatset} that the same sets are obtained.
\begin{corollary}\label{cor6}
Let $M\subseteq\Lbar$. Then $\sat(M)=\Lbar\cap \overline{\aff(M)}$. 
\end{corollary}
\begin{proof}
Since $\overline{\aff(M)}\subseteq \overline{\spn (M)}$ one has $\Lbar\cap \overline{\aff(M)}\subseteq \Lbar\cap\overline{\spn (M)}=\sat(M)$. It remains to show that $\Lbar\cap\overline{\operatorname{aff}(M)}\supseteq \sat(M)$. This is clear if $M=\emptyset$, so we may assume that $M$ is non-empty.

Let $\psi\in\Lbar\setminus \overline{\operatorname{aff}(M)}$. By the Hahn-Banach theorem, there exist $x\in X$ and $\gamma\in\RR$ such that $\psi(x)>\gamma$ and $\varphi(x)\le\gamma$ for all $\varphi\in\overline{\operatorname{aff}(M)}$. If $\varphi_1,\varphi_2\in\operatorname{aff}(M)$, then $\varphi_1+\alpha(\varphi_2-\varphi_1)\in\operatorname{aff}(M)$ for all $\alpha\in\RR$, hence $\varphi_1(x)=\varphi_2(x)$, as $\varphi(x)\le\gamma$ for all $\varphi\in\operatorname{aff}(M)$. Take $\varphi_0\in M$ and $z:=x-\varphi_0(x)u$. Then for every $\varphi\in\operatorname{aff}(M)$ we have $\varphi(x)=\varphi_0(x)$, so $\varphi(z)=\varphi(x)-\varphi_0(x)\varphi(u)=0$, since $\operatorname{aff}(M)\subseteq\Sigma$. In particular, $z\in Z(M)$. Also, $\psi(z)=\psi(x)-\varphi_0(x)\psi(u)=\psi(x)-\varphi_0(x)>\gamma-\gamma=0$, so that $\psi\not\in N(Z(M))$. Hence $\Lbar\cap \overline{\operatorname{aff}(M)}\supseteq N(Z(M))=\sat(M)$, according to 
Proposition \ref{charsatset}.
\end{proof}
As a preparation for the main result concerning the characterization of bands by means of subsets of $\Lbar$, we provide some lemmas.

\begin{lemma}\label{zm}
Let $B\subseteq X$. 
Then $Z(\car(B))=B^d$.
\end{lemma}

\begin{proof}
Let $x\in Z(\car(B))$, then for all $\varphi\in \car(B)$ one has $\varphi(x)=0$, which implies $\Phi(x)(\phi)=0$. Let $b\in B$. Then for all $\phi\in\Lbar\setminus \car(B)$ it holds that $\phi(b)=0$, so $\Phi(b)(\phi)=0$.
Hence $\Phi(b)\perp\Phi(x)$, so $b\perp x$. Consequently, $x\in B^d$.

Conversely, let $x\in B^d$, i.e., for all $b\in B$ one has $\Phi(b)\perp \Phi(x)$. Let $\phi\in \car(B)$, then there is $b\in B$ such that $\phi(b)\neq 0$. This implies $\Phi(b)(\phi)\neq 0$, so $\Phi(x)(\phi)=\phi(x)=0$. Consequently, $x\in Z(\car(B))$.  
\end{proof}

\begin{lemma}\label{band1}
Let $B\subseteq X$. Then $B\subseteq Z(N(B))\subseteq B^{dd}$.
\end{lemma}

\begin{proof}
If $x\in B$, then for all $\varphi\in N(B)$ it holds that $\varphi(x)=0$, hence $x\in Z(N(B))$.

Let $x\in Z(N(B))$ and $y\in B^d$. If $\varphi\in\Lbar \setminus N(B)$, then there is $b\in B$ such that $\varphi(b)\neq 0$, consequently $\Phi(b)(\varphi)\neq 0$, hence $\varphi(y)=\Phi(y)(\varphi)=0$. Thus, for each $\varphi\in \Lbar$
we have $\Phi(x)(\varphi)=\varphi(x)=0$ or $\Phi(y)(\varphi)=\varphi(y)=0$, therefore $x\perp y$. Hence $x\in B^{dd}$.
\end{proof}
As a direct consequence of Lemma \ref{band1}, we obtain the following result.

\begin{proposition} \label{prop2}
Let $B$ be a band in $X$. Then $B=Z(N(B))$. 
\end{proposition}
Therefore, in analogy to bands in $C(\Omega)$ (which were discussed in Section \ref{sec1}), a band in $X$ is determined whenever its carrier is known.

\begin{lemma}\label{Bd}
Let $B\subseteq X$ be such that $B=Z(N(B))$.
Then 
\[B^d=\{x\in X\colon \varphi(x)=0\mbox{ for all } \varphi\in \sat((N(B))^c)\}.\]
\end{lemma}
\begin{proof}
Due to Lemma \ref{zm}, one has 
\[B^d= Z(\car(B))=\{x\in X\colon \varphi(x)=0\mbox{ for all } \varphi\in (N(B))^c \}.\]
Let $x\in X$ be such that for all $\varphi\in (N(B))^c$ one has $\varphi(x)=0$, and let $\psi\in \sat((N(B))^c)=\Lbar\cap\overline{\spn ((N(B))^c)}$. Then it is straightforward that $\psi(x)=0$, hence the assertion is proved.
\end{proof}

The next definition introduces those saturated subsets of $\overline{\Lambda}$  the zero sets of which turn out to be bands. Just as a band and its disjoint complement, also such saturated subsets appear in pairs. The ensuing definition makes sense in the more general setting of a topological vector space $E$ (instead of $X'$) and a subset $V$ of $E$ (instead of $\Lbar$). With hindsight and in view of Corollary \ref{cor6}, we call a subset $M$ of $V$ \emph{saturated} (in $V$) if 
\[M=\sat(M):=V\cap\overline{\aff(M)}.\]

\begin{definition} \label{def.bisat}
Let $E$ be a topological vector space and let $V$ be a subset of $E$. A subset $M$ of $V$ is called \emph{bisaturated} (in $V$) if
\[M=V\cap \overline{\aff(V\setminus(V\cap \overline{\aff(V\setminus M)}))  }.\]
For two subsets $M_1$ and $M_2$ of $V$ the set $\{M_1,M_2\}$ is called a \emph{bisaturated pair} (in $V$) if
\[M_1=V\cap \overline{\aff(V\setminus M_2)} \mbox{ and } M_2=V\cap \overline{\aff(V\setminus M_1)}.\]
\end{definition}

For a subset $M$ of $V$ we abbreviate $M^c:=V\setminus M$. 
Note that a set $M\subseteq V$ is bisaturated if and only if
\[M=\sat((\sat(M^c))^c),\]
and for two subsets $M_1,M_2\subseteq V$ the set $\{M_1,M_2\}$ is a bisaturated pair if and only if $M_1=\sat({M_2}^c)$ and $M_2=\sat({M_1}^c)$. 

Clearly, every bisaturated subset of $V$ is saturated. For an example of a saturated set which is not bisaturated see Example \ref{ex.bands} below. Next we give an example of a bisaturated set.

\begin{example} Consider $X=C[0,1]$ (so $\Lbar$ consists of the point evaluations and can be identified with $[0,1]$), and \[B=\left\{x\in C[0,1]\colon x(t)=0 \mbox{ for all } t\in\left[\textstyle\frac{1}{2},1\right]\right\}.\]
We claim that the set $N(B)=\left[\textstyle\frac{1}{2},1\right]$
is bisaturated. Indeed, \[\sat ((N(B))^c)=\left[\textstyle 0,\frac{1}{2}\right],\] so 
\[\sat[(\sat [(N(B))^c])^c]=\left[\textstyle\frac{1}{2},1\right]=N(B).\]
Observe that $(N(B))^c=\left[0,\textstyle\frac{1}{2}\right)$
is not saturated, since every saturated set is closed.
\end{example}

The next lemma relates bisaturated sets and bisaturated pairs in $V$, where $V$ is a subset of a topological vector space $E$, and provides a geometric description of bisaturated sets. Observe that for a bisaturated pair $\{M_1,M_2\}$ in $V$ we have 
\begin{equation}\label{M1M2isall}
M_1\cup M_2=V,
\end{equation} 
since  $M_1\cup M_2= M_1\cup \sat(M_1^c)\supseteq M_1\cup {M_1}^c=V$.

\begin{lemma}\label{lemma.UB}
For two saturated sets $M_1,M_2\subseteq V$ the following four statements are equivalent.
\begin{itemize}
\item[\textrm{(a)}] $M_1$ is bisaturated and $M_2=\sat({M_1}^c)$;
\item[\textrm{(b)}] $\{M_1, M_2\}$ is a bisaturated pair;
\item[\textrm{(c)}] $\overline{\aff(M_1)}=\overline{\aff(V\setminus M_2)}$ and 
$\overline{\aff(M_2)}=\overline{\aff(V\setminus M_1)}$;
\item[\textrm{(d)}] there exist affine subspaces $S_1$ and $S_2$ of $E$ such that 
\[V\subseteq S_1\cup S_2,\ S_1=\overline{\aff(V\setminus S_2)}, \ S_2=\overline{\aff(V\setminus S_1)},\ M_i=V\cap S_i,\ i=1,2.\]
\end{itemize}
\end{lemma}
\begin{proof}
(a)$\Rightarrow$(b): $M_1$ is bisaturated, so $M_1=\sat((\sat({M_1}^c))^c)=\sat({M_2}^c)$. 

(b)$\Rightarrow$(c): We have 
\[M_2=V\cap\overline{\aff({M_1}^c)}\subseteq \overline{\aff({M_1}^c)},\]
 so $\overline{\aff(M_2)}\subseteq \overline{\aff({M_1}^c)}$. Also, 
\[M_2=V\cap \overline{\aff({M_1}^c)}\supseteq V\cap {M_1}^c={M_1}^c,\] 
so $\overline{\aff(M_2)}\supseteq\overline{\aff({M_1}^c)}$. Hence $\overline{\aff(M_2)}=\overline{\aff({M_1}^c)}$. Similarly, $\overline{\aff(M_1)}=\overline{\aff({M_2}^c)}$.

(c)$\Rightarrow$(d): Choose $S_i=\overline{\aff(M_i)}$, $i=1,2$. Since $M_i$ is saturated, $M_i=V\cap\overline{\aff(M_i)}=V\cap S_i$. Hence 
\[S_1=\overline{\aff(M_1)}=\overline{\aff(V\setminus M_2)}=\overline{\aff(V\setminus (V\cap S_2))}=\overline{\aff(V\setminus S_2)}\]
and, similarly, $S_2=\overline{\aff(V\setminus S_1)}$.

(d)$\Rightarrow$(a): Since $V\setminus M_1=V\setminus (V\cap S_1)=V\setminus S_1$, we have
\begin{eqnarray*}
M_1&=&V\cap S_1
=V\cap\overline{\operatorname{aff}(V\setminus S_2)}
=V\cap\overline{\operatorname{aff}\left(V\setminus\overline{\operatorname{aff}(V\setminus S_1)}\right)}\\
&=& V\cap\overline{\operatorname{aff}\left(V\setminus\overline{\operatorname{aff}(V\setminus M_1)}\right)}
=V\cap\overline{\operatorname{aff}\left(V\setminus\left(V\cap\overline{\operatorname{aff}(V\setminus M_1)}\right)\right)}\\&=&\sat((\sat({M_1}^c))^c),
\end{eqnarray*}
thus $M_1$ is bisaturated.
\end{proof}

For $E=X'$ and $V=\Lbar$, we relate bisaturated pairs in $\Lbar$ to disjoint complements in $X$.

\begin{proposition}\label{bipairs-disco}
Let $M_1,M_2\subseteq\Lbar$ be saturated sets and $B_1=Z(M_1)$ and $B_2=Z(M_2)$. Then $\{M_1,M_2\}$ is a bisaturated pair if and only if $B_1={B_2}^d$ and $B_2={B_1}^d$.
%
\end{proposition}
\begin{proof}
Let $\{M_1,M_2\}$ be a bisaturated pair.
To show that $B_2\subseteq {B_1}^d$, let $x_1\in B_1$, $x_2\in B_2$. For all $\varphi\in M_1$ one has $\Phi(x_1)(\varphi)=0$ and for all 
$\varphi\in M_2$ one has $\Phi(x_2)(\varphi)=0$. Since $M_1\cup M_2\supseteq \Lbar$, it follows that $\Phi(x_1)\perp \Phi(x_2)$, which implies $x_1\perp x_2$. Hence $B_2\subseteq {B_1}^d$.

Vice versa, let $x\in {B_1}^d$. Then for every $y\in B_1$ we have $x\perp y$ and hence $\Phi(x)\perp \Phi(y)$. Therefore, $\Phi(x)(\varphi)=0$ for every \[\varphi\in \car(B_1)=\Lbar\setminus N(B_1)=\Lbar\setminus M_1,\]
due to Proposition \ref{charsatset}. 
Thus, $\Phi(x)(\varphi)=\varphi(x)=0$ for every 
\[\varphi\in\sat(\car(B_1))=\sat(\Lbar\setminus M_1)=M_2.\]
Hence $x\in Z(M_2)=B_2$. Therefore ${B_1}^d=B_2$ and, by symmetry, ${B_2}^d=B_1$.

If $B_2={B_1}^d$, then due to Proposition \ref{charsatset},
\[\Lbar\setminus M_1=\Lbar\setminus N(Z(M_1))=\Lbar\setminus N(B_1)=\car(B_1),\]
so, by Lemma \ref{zm},
\[Z(\Lbar\setminus M_1)=Z(\car(B_1))={B_1}^d.\]
Again by Proposition \ref{charsatset},
\[M_2= N(Z(M_2))=N(B_2)=N({B_1}^d)=N(Z(\Lbar\setminus M_1))=\sat(\Lbar\setminus M_1).\]
By symmetry, $\{M_1,M_2\}$ is a bisaturated pair.
\end{proof}

We arrive at the main result in this section. The assumption in  the theorem is natural due to Proposition \ref{prop2}.

\begin{theorem}
\label{mainthm}
Let $B\subseteq X$ be such that $B=Z(N(B))$. Then $B$ is a band if and only if $N(B)$ is bisaturated.
\end{theorem}
\begin{proof}
Suppose that $B_1:=B=Z(N(B))$ is a band. Define $M_1:=N(B)$ and $M_2:=\sat({M_1}^c)$. Then by Lemma \ref{Bd}, $B_2:={B_1}^d=Z(M_2)$. Hence by Proposition \ref{bipairs-disco}, $\{M_1,M_2\}$ is a bisaturated pair. So $N(B)$ is a bisaturated set, according to Lemma \ref{lemma.UB}.

Conversely, let $M_1:= N(B)$ be bisaturated and let $B_1:=B=Z(N(B))$. Let $M_2:=\sat({M_1}^c)$. Then, by Lemma \ref{lemma.UB}, $\{M_1,M_2\}$ is a bisaturated pair. Let $B_2:= Z(M_2)$. Proposition \ref{bipairs-disco} yields that $B_1={B_2}^d={B_1}^{dd}$, hence $B=B_1$ is a band. 
\end{proof}


In Section \ref{sec4} we will use Theorem \ref{mainthm} to bound the number of bands in finite dimensional partially ordered vector spaces.

\section{Extension of bands}\label{sec3}

Let $X$ be a pre-Riesz space and let $(Y,i)$ be a vector lattice cover of $X$, for instance its Riesz completion. 
Recall that $u\perp v$ in $X$ is equivalent to $i(u)\perp i(v)$ in $Y$,
such that for $M\subseteq X$ we have $i(M^d)=i(M)^d\cap i(X)$. 

If $B$ is a band in $X$, then there are two natural procedures to 
construct a band $\widehat{B}$ in $Y$ such that 
\begin{equation}\label{exbands}
\widehat{B}\cap i(X)=i(B).
\end{equation} 
They are as follows:
\begin{itemize}
\item[(I)] $\widehat{B}:=(i(B^d))^d$.

In \cite[Proof of Proposition 5.12]{gaanskalauch} it is shown that
\eqref{exbands} is satisfied.
\item[(II)] 
$\widehat{B}:=(i(B))^{dd}$.

In this case, the subsequent statement shows \eqref{exbands}.
\end{itemize}
\begin{proposition}
\begin{itemize}
\item[(a)] If $B\subseteq X$, then $i(B)^{dd}\cap i(X)\subseteq i(B^{dd})$.
\item[(b)] If $B$ is a band in $X$, then $i(B)^{dd}\cap i(X) = i(B)$.
\end{itemize}
\end{proposition}
\begin{proof} 
(a) Let $x\in i(B)^{dd}\cap i(X)$. For $y\in i(B^d)=i(B)^d\cap i(X)$ we have $x\perp y$, so $x\in i(B^d)^d\cap i(X)=i(B^{dd})$.


(b) Clearly, $i(B)\subseteq i(B)^{dd}\cap i(X)$. As $B=B^{dd}$, (a) completes the proof.
\end{proof}

The band $\widehat{B}$ in (II) is the smallest band in $Y$ extending $B$ in the sense of \eqref{exbands}. Indeed, if $F$ is a band in $Y$ with $F\cap i(X)= i(B)$, then $F^d\subseteq i(B)^d$, so $F=F^{dd}\supseteq i(B)^{dd}=\widehat{B}$.

\medskip
We further analyze procedures (I) and (II) in the case that $Y$ is given by the functional representation as a subspace of $C(\Lbar)$. $B$ and 
$\widehat{B}$
are determined by means of their carriers, so it is natural to ask how $\car(B)$ and $\car(\widehat{B})$
are related and how to determine $\car(\widehat{B})$ if $\car(B)$ is known.


For every band $F$ in $C(\overline{\Lambda})$ extending $B$ one has 
\[\car(B)\subseteq \car(\Phi(B)^{dd})\subseteq \car(F),\]
since $\Phi(B)^{dd}$ corresponding to (II) is the smallest band extending $B$.
In general, $\car(B)$ and $\car(\Phi(B)^{dd})$ may differ, as the next example shows.
\begin{example}\label{exa:namioka}
Let \[X=\left\{x\in C[-1,1]\colon x(0)=\textstyle\frac{1}{2}(x(-1)+x(1))\right\}.\]
The function $u=1$ is an order unit in $X$, so $\Lambda$ can be identified with $[-1,0)\cup(0,1]$, and $\Lbar$ with $[-1,1]$. 
Let \[B=\left\{x\in X\colon x(1)=0 \mbox{ and } x(t)=0 \mbox{ for all } t\in [-1,0] \right\}.\]
Then  \[B^d=\left\{x\in X\colon x(-1)=0 \mbox{ and } x(t)=0 \mbox{ for all } t\in [0,1] \right\}.\]
Furthermore, $B=B^{dd}$, so $B$ is a band in $X$ with $\car(B)=(0,1)$.
In $C(\Lbar)$, we obtain
\[\Phi(B)^d=\left\{x\in C(\Lbar)\colon x(t)=0 \mbox{ for all } t\in [0,1] \right\}\]
and \[\Phi(B)^{dd}=\left\{x\in C(\Lbar)\colon x(t)=0 \mbox{ for all } t\in [-1,0] \right\},\]
so $\car(\Phi(B)^{dd})=(0,1]\neq \car(B)$.  
\end{example}
Let us investigate procedure (I) in more detail.
To determine $\car(\Phi(B^d)^d)$ by means of $\car(B)$, we need the following identity.
\begin{lemma}\label{papagei}
For a subset $M\subseteq X$ one has 
$\car(\Phi(M)^d)=\overline{\car(M)}^c$.  
\end{lemma}
\begin{proof}
For $s\in C(\overline{\Lambda})$ one has $s\in \Phi(M)^d$ if and only if $s(\varphi)=0$ for all $\varphi\in \car(M)$, which is equivalent to  
$s(\varphi)=0$ for all $\varphi\in \overline{\car(M)}$. 
Hence $\car(\Phi(M)^d)\subseteq\overline{\car(M)}^c$. If $\psi\in \overline{\car(M)}^c$, then by Urysohn's lemma there exists $s\in C(\overline{\Lambda})$ such that $s(\psi)\neq 0$ and $s(\varphi)=0$ for all $\varphi\in \overline{\car(M)}$, so $s\in \Phi(M)^d$. Therefore
$\car(\Phi(M)^d)\supseteq\overline{\car(M)}^c$.
\end{proof}
\begin{theorem}\label{th:proc1}
For a band $B$ in $X$, the carrier of $\Phi(B^d)^d$ is given by 
\[\car(\Phi(B^d)^d)= \left(\overline{(\sat(\car(B)))^c}\right)^c,\]
and $\Phi(B^d)^d$ is the largest ideal with carrier $\left(\overline{(\sat(\car(B)))^c}\right)^c$.
\end{theorem}
\begin{proof}
Due to Lemma \ref{Bd} and \ref{twobasicprops} (ii), we have
\[B^d=Z(\sat(N(B)^c))=Z(N(B)^c).\]
Then Proposition \ref{charsatset} implies 
\[N(B^d)=\sat(N(B)^c)=\sat(\car(B)),\]
hence $\car(B^d)=(\sat(\car(B)))^c$. Applying Lemma \ref{papagei} with $M=B^d$ completes the proof. 
\end{proof}
We proceed by discussing procedure (II).
First we show that the carrier of $B$ and the carrier of $\widehat{B}$ have the same closure (in other words, the supports of $B$ and $\widehat{B}$ are the same).

\begin{proposition}\label{prop3}
For a band $B$ in $X$ one has $\overline{\car(B)}=\overline{\car(\Phi(B)^{dd})}$.
\end{proposition}

\begin{proof}
We know that $\car(B)\subseteq \car(\Phi(B)^{dd})$, so $\overline{\car(B)}\subseteq \overline{\car(\Phi(B)^{dd})}$.

Vice versa, suppose there is $\varphi_0\in \overline{\car(\Phi(B)^{dd})}\setminus \overline{\car(B)}$. By Urysohn's lemma, there is $x\in C(\Lbar)$ such that $x(\varphi_0)\neq 0$ and $x(\varphi)=0$ for all $\varphi\in \overline{\car(B)}$. If $b\in B$,
then $\Phi(b)(\varphi)=\varphi(b)=0$ for all $\varphi\in \Lbar\setminus \car(B)$, so $x\perp \Phi(b)$ for all $b\in B$.
Hence $x\in \Phi(B)^d$. Since $\varphi_0\in \overline{\car(\Phi(B)^{dd})}$ and 
$x(\varphi_0)\neq 0$, there exists $\varphi_1\in \car(\Phi(B)^{dd})$ such that $x(\varphi_1)\neq 0$. Moreover, there is $y\in \Phi(B)^{dd}$ such that
$y(\varphi_1)\neq 0$. This implies $x\not\perp y$, which contradicts $x\in \Phi(B)^d$, $y\in \Phi(B)^{dd}$. Thus $\overline{\car(B)}= \overline{\car(\Phi(B)^{dd})}$.
 \end{proof}
 Since $\Phi(B)^{dd}$ is a band in $C(\Lbar)$, the set $\car(\Phi(B)^{dd})$ is regularly open, i.e. $\car(\Phi(B)^{dd})=\operatorname{int}\left(\overline{\car(\Phi(B)^{dd})}\right)$. Now apply Proposition \ref{prop3} to obtain the following result.
\begin{theorem} \label{th:proc2}
For a band $B$ in $X$, the carrier of $\Phi(B)^{dd}$ is given by \[\car(\Phi(B)^{dd})=\operatorname{int}\left(\overline{\car(B)}\right),\]
and $\Phi(B)^{dd}$ is the largest ideal with carrier
$\operatorname{int}\left(\overline{\car(B)}\right)$.
\end{theorem}

\begin{remark}
$\car(\Phi(B)^{dd})$ is the smallest regularly open set containing $\car(B)$.
Indeed, $\operatorname{int}(\overline{\car(B)})$ is a regularly open set, $\car(B)\subseteq \overline{\car(B)}$ and $\car(B)$ is open, so $\car(B)\subseteq\operatorname{int}(\overline{\car(B)})$. Hence $\operatorname{int}(\overline{\car(B)})$ is a regularly open set containing $\car(B)$. Let $O$ be any regularly open set with $O\supseteq\car(B)$. Then $\overline{O}\supseteq \overline{\car(B)}$, so $O=\operatorname{int}( \overline{O}) \supseteq \operatorname{int}(\overline{\car(B)})$. So $\car(\Phi(B)^{dd})=\operatorname{int}(\overline{\car(B)})$ is the smallest regularly open set containing $\car(B)$.

In Example \ref{exa:namioka}, $\car(\Phi(B)^{dd})=(0,1]$ is indeed the smallest regularly open set in $\overline{\Lambda}=[-1,1]$ containing $\car(B)=(0,1)$. 
\end{remark}

For a band $B$ in $X$, we address the question whether the carrier of $B$ can be recovered if the carrier of $\Phi(B)^{dd}$ is known. 

Observe that every saturated subset of $\Lbar$ is w*-closed. On the contrary, a w*-closed subset of $\Lbar$ is not saturated, in general
(note that in Example \ref{ex.bands} the set $\{f_1,f_2,f_3\}$ is w*-closed but not saturated). Nevertheless, a statement similar to Proposition \ref{prop3} is true, concerning the saturations of $\car(B)$ and $\car(\Phi(B)^{dd})$.
\begin{proposition}\label{prop27}
For a band $B$ in $X$, one has 
\[\sat(\car(B))=\sat(\car(\Phi(B)^{dd})).\]
\end{proposition}
\begin{proof}
First observe that for sets $M,N\subseteq\overline{\Lambda}$ with $\overline{M}=\overline{N}$ it follows that $\sat(M)=\sat(N)$.
Indeed, the continuity of the vector space operations yields $\operatorname{span}\left(\overline{M}\right)\subseteq\overline{\operatorname{span}(M)}$, hence $\overline{\operatorname{span}\left(\overline{M}\right)}=\overline{\operatorname{span}(M)}$, which implies $\overline{\operatorname{span}(M)}=\overline{\operatorname{span}(N)}$.

Now the statement follows immediately from Proposition \ref{prop3}.
%
%
\end{proof}
Next we reconstruct $\car(B)$ provided $\car(\Phi(B)^{dd})$ is known.
\begin{proposition}
For a band $B$ in $X$, one has
\[\car(B)=(\sat((\sat(\car(\Phi(B)^{dd})))^c))^c.\]
\end{proposition}
\begin{proof}
Due to Theorem \ref{mainthm}, $(\car(B))^c$ is bisaturated.
By Proposition \ref{prop27} we obtain 
\[(\car(B))^c=\sat((\sat(\car(B)))^c)=\sat((\sat(\car(\Phi(B)^{dd})))^c),\]
which implies the assertion.
\end{proof}
 
Summing up the above results, the extension method (II) is such that $\car(B)$ and $\car(\Phi(B)^{dd})$ have the same closure and the same saturation, and if one of the two sets is known then the other one is determined.

Where procedure (II) yields the smallest band $\widehat{B}$ in $Y$
extending a band $B$ in $X$, there is no largest one, in general, as the next example shows.

\begin{example} 
Let $S$ be the unit circle in $\R^2$ and let $X$ be the space of restrictions to $S$ of affine functions from $\R^2$ to $\R$ with pointwise order on $S$. $C(S)$ is a vector lattice cover of $X$ \cite[Example 2.8]{gaans2003}. Let $O_1:=\{(a,b)\in S \colon a>0, \, b>0\}$ and $O_2:=\{(a,b)\in S \colon a<0, \, b<0\}$. Then $O_1$ and $O_2$ are regularly open subsets of $S$, therefore $I_{O_1}$ and $I_{O_2}$ are bands in $C(S)$ by Proposition \ref{pro2}.
Affine functions in $I_{O_1}$ or $I_{O_2}$ are identically  zero, hence 
both $I_{O_1}$ and $I_{O_2}$ extend the band $\{0\}$ in $X$. The only band in $C(S)$ larger than $I_{O_1}$ and $I_{O_2}$ is $C(S)$, which does not extend $\{0\}\subseteq X$. 
\end{example}

\section{Bands in finite dimensional spaces}\label{sec4}
 In this section we shall see how to use the characterization of bands in terms of bisaturated sets of $\overline{\Lambda}$, given in Theorem \ref{mainthm}, to analyze the following question. 
 \begin{question}
Does there exist an a-priori upper bound for the number of bands in a finite dimensional  
partially ordered vector space $X$ with a closed and generating cone $K$ in terms of the dimension of $X$? If so, what is the optimal upper bound? 
 \end{question}
Clearly the number of bands in $n$-dimensional partially ordered vector space is $2$ if $n=1$. For $n\geq 2$ we have the following result. 
\begin{theorem}\label{thm:numbands}
The number of bands in an $n$-dimensional partially ordered vector space $(X,K)$, where is $K$ is a closed generating cone and $n\geq 2$, does not exceed
\[
\textstyle\frac{1}{4}2^{2^{n}}.
\] 
\end{theorem}
 The proof of this result consists of two parts. First it will be shown that the upper bound holds for partially ordered vector spaces with polyhedral cones. Subsequently we use  the Plastering Lemma \ref{plasteringlemma} to show that the same upper bound holds for general cones. 

Let $X$ be a finite dimensional vector space. For a subset $M\subseteq X$ we mean by $\dim M$ the dimension of $\aff M$.
Recall that a closed cone $K$ in $X$ is {\em polyhedral} if it is the intersection of finitely many closed half-spaces. We use basic results in polyhedral geometry, which can be found, e.g., in \cite{Ziegler}. If $K$ is a closed polyhedral cone in $X$ with non-empty interior, then each $u$ in the interior of $K$ is an order unit. Fix such a $u$. Then there exist $m\geq \dim X$ functionals, $\phi_1,\ldots, \phi_m\in X'$ with $\phi_i(u)=1$ for all $i$, so that 
 \begin{equation}\label{f.polycone}
K=\{x\in X\colon \phi_i(x)\geq 0\mbox{ for }i=1,\ldots,m\}.
 \end{equation}
The functionals $\phi_i$ define the facets of $K$. Recall that  a {\em facet} of a polyhedral cone $K$ is an (exposed) face $F=\{x\in K\colon \psi(x)=0\}$, where $\psi\in K'$, such that $\dim F =\dim K-1=\dim X-1$.  Observe that 
\[
 K'=\left\{\phi\in X'\colon \phi(x)\geq 0\mbox{ for all }x\in K\right\}=\left\{\mbox{$\sum_i \alpha_i\phi_i$}\colon \alpha_i\geq 0\mbox{ for all }i\right\}.
 \]
Indeed,  $\{\sum_i \alpha_i\phi_i\colon \alpha_i\geq 0\mbox{ for all }i\}\subseteq K'$. Now if there exists $\psi\in K'$ with $\psi\not\in \{\sum_i \alpha_i\phi_i\colon \alpha_i\geq 0\mbox{ for all }i\}$, then we can use the Hahn-Banach separation theorem to find $z\in X''=X$ such that $\psi(z)<0$ and $\phi_i(z)\geq 0$ for all $i$. As $K=\{x\in X\colon \phi_i(x)\geq 0\mbox{ for all }i\}$, we see that $z\in K$, which is impossible, as $\psi\in K'$ and $\psi(z)<0$. 

Let $\Lambda$ be the set of extreme points  of 
$\Sigma :=\{\phi\in K'\colon \phi(u)=1\}$. We claim that $\Lambda =\{\phi_1,\ldots,\phi_m\}$, which is a closed set. Indeed, if there exist $\sigma$ and $\rho$ in $\Sigma$ such that $\lambda \sigma +(1-\lambda)\rho = \phi_i$ for $0<\lambda<1$ and some $i$, then $F=\{x\in K\colon \phi_i(x)=0\}$ is contained in $\ker \sigma\cap \ker \rho$. As $\dim F = \dim X-1$, we find that $\sigma =\alpha \phi_i$ and $\rho=\beta \phi_i$ for some 
$\alpha,\beta\in\mathbb{R}$. It now follows that $\sigma=\phi_i=\rho$, as $\sigma(u)=\phi_i(u)=\rho(u)$.

The space $X$ with polyhedral cone $K$ as defined above fits in the setting of the functional representation given in Section \ref{sec1}.  As $K$ is closed,  $X$ is Archimedean. By Theorem \ref{thm:main},  
\[
\Phi\colon(\RR^n,K)\rightarrow C(\Lambda),\quad
x\mapsto (\phi_1(x),\ldots,\phi_m(x)), 
\]
is bipositive and 
$\Phi(\RR^n)$
is order dense in $C(\Lambda)$, where $C({\Lambda})$ is identified with $\RR^m$ ordered by the standard cone $\mathbb{R}^m_+$. By
\cite[Proposition 13]{KalLemGaa2010}, $(\mathbb{R}^m,\Phi)$ is even the Riesz completion of $X$.

Observe that for a set $M\subseteq \Lambda$
\[\sat(M)=\Lambda\cap\aff (M).\]
\begin{remark}\label{rem.indexed}
Since every subset of $\Lambda$ is regularly open, we know that
 $\widehat{B}$ is a band in the vector lattice $\RR^m$ ordered by the standard cone if and only if there exists $N\subseteq \{1,\ldots,m\}$ such that 
\begin{equation}
\widehat{B}=\left\{(x_1,\ldots,x_m)^T\colon x_i=0 \mbox{ for all } i\in N\right\}.
\end{equation}
Moreover, by Theorem \ref{th:proc2}, a band $B$ in $X$ satisfies 
$\car(B)=\car(\Phi(B)^{dd})$.
\end{remark}

Next we state our main result on the number of bands in spaces with polyhedral cones.

\begin{theorem}\label{polybands}
The number of bands in an $n$-dimensional partially ordered vector space $X$ ordered by a generating polyhedral cone $K$ does not exceed
$\frac{1}{4}2^{2^{n}}$, provided $n\ge 2$. 
\end{theorem}
By using Theorem \ref{mainthm}, the proof of Theorem \ref{polybands} reduces to a purely combinatorial problem concerning polytopes. Indeed, note that in the setting above, the set  $\Sigma$ is a closed convex polytope in the $(n-1)$-dimensional affine hyperplane $H=\{\phi\in X'\colon \phi(u)=1\}$ with vertex set $V=\{\phi_1,\ldots,\phi_m\}$. From  Theorem \ref{mainthm} we know that $B$ is a band in $X$ if and only if $N(B)$ is a bisaturated set in $\Lambda$. Therefore the number of bands equals the number of bisaturated sets.

Given a  finite set $V$ in $X$, Definition \ref{def.bisat} entails that a set $S\subseteq V$ is bisaturated in $V$ if and only if 
\[
S = V\cap\aff (V\setminus(V\cap\aff(V\setminus S))), 
\]
and $\{S,T\}$ is a  bisaturated pair in $V$ if and only if
\[
S=V\cap\aff(V\setminus T)\mbox{\quad and \quad}T=V\cap\aff(V\setminus S). 
\]
Thus, to prove Theorem \ref{polybands} it suffices to show the following combinatorial statement. 
\begin{theorem}\label{polytopes}
In a vertex set $V$ of a polytope $P$ in $\mathbb{R}^{n-1}$, with $n\geq 2$, there are at most $\frac{1}{4}2^{2^{n}}$ bisaturated sets. 
\end{theorem}
\begin{proof}
Suppose that there exists a polytope $P$ in $\mathbb{R}^{n-1}$ with $q$ bisaturated sets in the vertex set $V$ of $P$.  We wish to show that $q\leq \frac{1}{4}2^{2^n}$. Without loss of generality we may assume that $P$ is chosen such that the number of its vertices is minimal among the cardinalities of the vertex sets of polytopes in $\mathbb{R}^{n-1}$ with at least $q$ bisaturated sets. \\
\mbox{}\\
{\em Claim 1:} For each $k$-dimensional affine subspace $U$ in $\mathbb{R}^{n-1}$, with $1\leq k\leq n-1$, we have that 
\[|V\cap U|\leq 2^{k +1}-2.\] 

Assuming the claim for the moment, we see that $V$ has at most $2^{n}-2$ points, so 
the number of bisaturated sets of $P$ cannot exceed $2^{2^{n}-2}$. So, it remains to prove the claim. Obviously, the claim is true for $k=1$, as no three vertices of $P$ lie 
on a line. 

For the sake of contradiction, assume that there exists a $k$-dimensional  affine subspace $U$, with $k\geq 2$, such that 
\[|U\cap V|>2^{k+1}-2.\] 
Now take 
$k\geq 2$  as small as possible. 
\\
\mbox{}\\
{\em Claim 2:}
For each $W\subseteq V$ with $|W\cap U|> 2^k -2$, we have that $U\subseteq \aff W$.

For a proof of Claim 2, observe that $U\not\subseteq \aff W$ implies that $U':=U\cap \aff W$ has dimension $l< k$ and $|U'\cap V|> 2^{k}-2\ge 2^{l+1}-2$, which would mean that $k$ is not minimal. 

We continue the proof of Claim 1 and fix $u\in V\cap U$. To arrive at a contradiction we show that the polytope $P'$ with vertex set $V'=V\setminus\{u\}$ has the property that  if $S$ is a bisaturated set in $V$, then $S'=S\setminus\{u\}$ is a bisaturated set in  $V'$.  Moreover, $V'$ contains at least as many bisaturated sets as $V$. The last statement contradicts the minimality assumption on the number of  vertices of $P$.

Let $\{S,T\}$ be a bisaturated pair of $V$, and write $S'=S\setminus\{u\}$ and $T'=T\setminus\{u\}$. Then since $S\cup T=V$, either $|S\cap U|>2^k -1$ or $|T\cap U|>2^k-1$. Without loss of generality we may assume  that $|S\cap U|> 2^k-1$. Then $U\cap V\subseteq S$, as $U\subseteq \aff S$ (by Claim 2) and $V\cap \aff S=S$, and hence $u\in S$. 
This implies that $V\setminus S=V'\setminus S'$; so,
\[
T' = T\setminus\{u\} =V\cap \aff(V\setminus S)\setminus \{u\}  =V'\cap \aff (V\setminus S)=V'\cap \aff (V'\setminus S'). 
\]
To show that $S' = V'\cap \aff (V'\setminus T')$ we distinguish two cases: $ u\in T$ and $u\not\in T$. 
If $u\in T$, then 
\[
S' = S\setminus\{u\} =   V'\cap   \aff (V\setminus T) = V'\cap \aff(V'\setminus T').
\]
On the other hand, if $u\not\in T$, then  $|T\cap U|\leq 2^k -2$. Indeed, $|T\cap U|>2^k -2$ by means of Claim 2 implies $U\subseteq \aff T$, so that  $U\cap V\subseteq V\cap \aff T=T$, which 
shows that $u\in T$. Thus $|(V\setminus T)\cap U|\ge 2^k$, which implies that $|(V'\setminus T')\cap U| > 2^k-2$. Again by Claim 2, it follows that  
$U\subseteq  \aff (V'\setminus T')$; in particular, $u\in \aff (V'\setminus T')$. So, 
\[
S' = S\setminus \{u\} =V'\cap \mathrm{aff}(V\setminus T) = V'\cap \mathrm{aff}(V'\setminus T').
\]
We conclude that $\{S',T'\}$ is a bisaturated pair of $P'$. 
 
It remains to show that $P'$ has at least as many bisaturated sets as $P$. 
Recall that for each bisaturated set $S$ of $P$, $S'=S\setminus \{u\}$ is a bisaturated set of $P'$. So, it suffices to show that if $S$ is a bisaturated set of $P$ and $u\in S$, then $S'$ is not a bisaturated set of $P$. Assume that both $S$ and $S'$ are bisaturated sets of $P$. 
Note that if $|S\cap U|>2^k-1$, then $|S'\cap U|> 2^k-2$. But this implies (by Claim 2) that $U\subseteq \aff(S')$, so $V\cap U\subseteq V\cap \aff S'=S'$, which is impossible, as $u\not\in S'$. 
Now assume that $|S\cap U|\leq 2^k-1$; so $|(V\setminus S)\cap U|\ge 2^k> 2^k-1$ and then (by Claim 2 for the last time) $U \subseteq \aff(V\setminus S)$. 
Let $T = V\cap \aff (V\setminus S)$ and $R =V\cap \aff (V\setminus S')$.
Obviously, if we can show that  $T=R$, then it follows that $S =S'$ and we are done. 
Note that, as $u\in \aff(V\setminus S)$, $\aff(V\setminus S) = \aff(V\setminus S')$,  
and hence $T=R$. 
\end{proof}

We next generalize the upper bound in Theorem \ref{polybands} for the polyhedral case to the general finite dimensional case. Let $X$ be a finite dimensional vector space and $K$ a generating closed cone in $X$. We denote by $(X,K)$ the partially ordered vector space $X$ ordered by $K$. We fix an interior point $u$ of $K$ and consider the corresponding functional representation $(\Phi,\Lbar)$ as in Section \ref{sec1}.
Given bands $B_1,B_2,\ldots,B_q$ in $(X,K)$, the idea now is to construct a  polyhedral cone $L$ in $X$ containing $K$ such that each $B_i$ is a band in $(X,L)$. The construction uses the characterisation of bands by means of bisaturated sets of 
$\Lbar$ and  the following  two technical lemmas.

If $\{M_1,M_2\}$ is a bisaturated pair in $\Lbar$ and we replace $M_1$ and $M_2$ by smaller closed sets $W_1$ and $W_2$ that have the same affine hulls as $M_1$ and $M_2$, respectively, then the next lemma says that the pair of their saturations in $W=W_1\cup W_2$ is a bisaturated pair in $W$.

\begin{lemma}\label{lemma.bipairs-restricted}
Let $\{M_1,M_2\}$  be a bisaturated pair in $ \Lbar$. If $W_1\subseteq \Lbar\setminus M_2$ and $W_2\subseteq \Lbar\setminus M_1$ are such that
\[\overline{\aff(W_i)}=\overline{\aff(M_i)},\mbox{\quad for }i=1,2,\]
and if $F$ is a closed subset of $\Lbar$ with $W_1\cup W_2\subseteq F$, then for the sets
\[N_i:=F\cap \overline{\aff(W_i)},\mbox{\quad with } i=1,2,\]
the pair $(N_1, N_2)$ is bisaturated in $F$.
\end{lemma}
\begin{proof}
We have
\begin{align*}
&F\setminus N_1=F\setminus(F\cap\overline{\aff(W_1)})
 = F\cap(\Lbar\setminus \overline{\aff(W_1)})\\
&\qquad =F\cap(\Lbar\setminus(\Lbar\cap\overline{\aff(M_1)}))=F\cap(\Lbar\setminus M_1)=F\setminus M_1,
\end{align*}
as $M_1$ is saturated in $\Lbar$, so
\[\overline{\aff(F\setminus N_1)}=\overline{\aff(F\setminus M_1)}.\]
On the one hand, $F\setminus M_1\subseteq \Lbar\setminus M_1$; so,
\[\overline{\aff(F\setminus M_1)}\subseteq \overline{\aff(\Lbar\setminus M_1)}=\overline{\aff(M_2)}=\overline{\aff(W_2)},\]
by Lemma \ref{lemma.UB}. On the other hand,
\[F\setminus M_1=F\cap (\Lbar\setminus M_1)\supseteq F\cap W_2=W_2,\]
so that $\overline{\aff(F\setminus M_1)}\supseteq \overline{\aff(W_2)}$. Thus
\[\overline{\aff(F\setminus N_1)}=\overline{\aff(F\setminus M_1)}=\overline{\aff(W_2)}\]
and therefore
\[F\cap\overline{\aff(F\setminus N_1)}=F\cap\overline{\aff(W_2)}=N_2.\]
Similarly, $F\cap\overline{\aff(F\setminus N_2)}=N_1$.
\end{proof}

\begin{lemma}\label{lemma.no-bar-needed}
If $B$ is a subspace of a finite dimensional vector space, then
\[
\aff(\car(B))=\aff\{\phi\in\Lambda\colon \phi(x)\neq 0\mbox{ for some }x\in B\}.
\]
\end{lemma}
\begin{proof}
Denote $C=\aff\{\phi\in\Lambda\colon \phi(x)\neq 0\mbox{ for some }x\in B\}$. 
Since $\Lambda\subseteq \Lbar$, $\aff(C)\subseteq\aff(\car(B))$. The affine subspace $\aff(C)$ is finite dimensional, hence closed, so it contains the closure of $C$. It remains to show that $\car(B)\subseteq \overline{C}$.

If $\phi\in\Lbar$ and $x\in B$ are such that $\phi(x)\neq 0$, then there are $(\phi_\alpha)_\alpha$ in $\Lambda$ with $\phi_\alpha\to\phi$ (in weak* topology) and then $\phi_\alpha(x)\neq 0$ for $\alpha$ large. Therefore $\phi_\alpha\in C$ for $\alpha$ large, so $\phi\in\overline{C}$. Thus $\aff(\car(B))\subseteq \aff(\overline{C})=\aff(C)$.
\end{proof}

Now we are in a position to present the plastering lemma.

\begin{lemma}\label{plasteringlemma}
Let $(X,K)$ be a finite dimensional partially ordered vector space with a closed generating cone $K$. For any choice of bands $B_1,\ldots,B_q$ in $(X,K)$, there exists a closed generating polyhedral cone $L$ in $X$ with $K\subseteq L$ such that $B_i$ is a band in $(X,L)$ for all  $i=1,\ldots,q$.
\end{lemma}
\begin{proof}
Due to Proposition \ref{prop2}, $B_i=Z(N(B_i))$ and ${B_i}^d=Z(N({B_i}^d))$ and then according to Proposition \ref{charsatset} the sets $N(B_i)$ and $N({B_i}^d)$ are saturated. By Proposition \ref{bipairs-disco}, it follows that $\{N(B_i),N({B_i}^d)\}$ is a bisaturated pair in $\Lbar$. By definition,
\[\car(B_i)=\Lbar\setminus N(B_i)\mbox{ and }\car({B_i}^d)=\Lbar\setminus N({B_i}^d).\]
Choose, with the aid of Lemma \ref{lemma.no-bar-needed}, finite sets
\[U_i\subseteq\car(B_i)\cap\Lambda\mbox{ and }V_i\subseteq\car({B_i}^d)\cap\Lambda\] 
such that $\aff(U_i)=\aff(\car(B_i))$ and $\aff(V_i)=\aff(\car({B_i}^d))$.
Define
\[F:=\bigcup_{i=1}^N (U_i\cup V_i)\]
and $L:=\{x\in X\colon\, \phi(x)\ge 0\mbox{ for every }\phi\in F\}$.
Since for each $i$, $\{N(B_i),N({B_i}^d)\}$ is a bisaturated pair, we obtain
\begin{align*}
&\Lbar\cap \aff(U_i\cup V_i)\supseteq \Lbar\cap \aff(U_i)=\Lbar\cap \aff(\car(B_i))\\
&\qquad =\Lbar\cap \aff(\Lbar\setminus N(B_i))=N({B_i}^d)
\end{align*}
and, similarly, $\Lbar\cap\aff(U_i\cup V_i)\supseteq N(B_i)$, which implies \[\Lbar=N(B_i)\cup N({B_i}^d)\subseteq \Lbar\cap \aff(U_i\cup V_i),\] 
due to \eqref{M1M2isall}. Hence $\Lbar\subseteq\aff(F)$, so that $L$ is a cone. Also, since $F\subseteq\Lambda$, we have $L\supseteq K$, so that $L$ is generating. Moreover, the order unit $u$ of $(X,K)$ is an interior point of $K$, hence an interior point of $L$ and therefore an order unit in $(X,L)$.

Every element of $F$ is an extreme point of
\[\Sigma_L=\{\phi\in X'\colon\, \phi(x)\ge 0\mbox{ for all }x\in L,\ \phi(u)=1\}.\]
Indeed, if $\phi\in F$, $0<\lambda<1$, and $\phi_1,\phi_2\in \Sigma_L$ are such that $\phi=\lambda\phi_1+(1-\lambda)\phi_2$, then $\phi\in\Lambda$ and $\phi_1,\phi_2\in\Sigma$, so $\phi_1=\phi_2=\phi$, as $\Lambda$ consists of extreme points of $\Sigma$. 
In conclusion, $L$ is a generating polyhedral cone in $X$, $u$ is an order unit in $(X,L)$, and $F$ is the set of extreme points of $\Sigma_L$. 

It remains to show that for each $i$ the subspace $B_i$ is a band in $(X,L)$. For $1\le i\le N$ we have 
\begin{align*}
&U_i\subseteq\car(B_i)=\Lbar\setminus N(B_i),\\
&V_i\subseteq\car({B_i}^d)=\Lbar\setminus N({B_i}^d),
\end{align*}
and, since $\{N({B_i}^d),N(B_i)\}$ is a bisaturated pair in $\Lbar$,
\begin{align*}
&\aff(U_i)=\aff(\car(B_i))=\aff(\Lbar\setminus N(B_i))=\aff N({B_i}^d),\\
&\aff(V_i)=\aff(N(B_i)).
\end{align*}
According to Lemma \ref{lemma.bipairs-restricted}, the saturations in $F$ of the sets $U_i$ and $V_i$
constitute a bisaturated pair in $F$, so that, according to Proposition \ref{bipairs-disco}, $Z(U_i)$ and $Z(V_i)$ are bands in $(X,L)$ and $Z(U_i)^d=Z(V_i)$. 

Finally, 
by Lemma \ref{twobasicprops}(ii),
\[Z(U_i)=Z(\sat(U_i))=Z(\sat(\car(B_i)))=Z(\sat(\Lbar\setminus N(B_i))),\]
hence $Z(U_i)={B_i}^d$ by Lemma \ref{Bd}. Similarly, $Z(V_i)=B_i$. It follows that every $B_i$ is a band in $(X,L)$.
\end{proof}

A combination of Theorem \ref{polybands} and Lemma \ref{plasteringlemma} yields 
Theorem \ref{thm:numbands}. We believe that the upper bound  in Theorem \ref{thm:numbands} is far from optimal. The following example shows that there exist 
$(n+1)$-dimensional partially ordered vector spaces with ${ 2n\choose n} +2$ bands. 
In particular, it follows  that there exist partially ordered vector spaces with dimension $n\geq 4$ that have more bands than an $n$-dimensional Archimedean vector lattice, which has $2^n$ bands.  
\begin{example}\label{nr_bands}
Let $P$ be a polytope in $\mathbb{R}^n$ with vertex set $V=\{v_1,\ldots,v_{2n}\}$ that is in general linear position, i.e., for each affine subspace $U\subset \mathbb{R}^n$ with $\dim U<n$ we have that 
$|V\cap U|\leq \dim U +1$. In other words, there are no 3 points in $V$ on a line, no 4 points in $V$ in a plane, etc. To show existence of such a set $V$, let  $n\geq 1$ and $S_{n-1}$ be the unit sphere in $\mathbb{R}^n$. Choose uniformly at random $v_1,\ldots,v_{2n}$ from $S_{n-1}$. Then the convex hull of $\{v_1,\ldots,v_{2n}\}$ is, with probability 1, a polytope, $P$, with non-empty interior. If necessary apply a  translation, so that $0$ is in the interior of $P$. By construction, the vertex set 
$V= \{v_1,\ldots,v_{2n}\}$ of $P$ is in general linear position.

For $i=1,\ldots,2n$ let $w_i=(v_i,1)\in\mathbb{R}^n\times\mathbb{R}=\mathbb{R}^{n+1}$ and define 
\[
K_V=\{x\in\mathbb{R}^{n+1}\colon \langle w_i,x\rangle\geq 0 \mbox{ for all }i=1,\ldots,2n\}.
\]
Here $\langle\cdot,\cdot\rangle$ denotes the usual inner-product on $\mathbb{R}^{n+1}$. 
As $w_1,\ldots,w_{2n}$ span $\mathbb{R}^{n+1}$, $K_V$ is a closed generating polyhedral cone in $\mathbb{R}^{n+1}$ with $u=(0,1)\in\mathbb{R}^n\times\mathbb{R}$ in its interior.  Now note that 
\[
\Sigma =\{z\in K_V'\colon \langle z,u\rangle =1\} =
\{z\in K_V'\colon z=(v,1)\mbox{ and }v\in P\}
\] 
and $\Lambda =\{w_1,\ldots,w_{2n}\}$, as $V$ is in general linear position.
Moreover, if $U$ is an affine subspace of $\mathbb{R}^{n+1}$ contained in $\Sigma$ with $\dim U<n$, then $|\Lambda\cap U|\le \dim U+1$. 

Next we prove that $(\mathbb{R}^{n+1},K_V)$ has ${2n\choose n}+2$ bands by showing that $S\subseteq\Lambda$ is bisaturated if and only if $S=\emptyset$, $S=\Lambda$, or $|S|=n$. 
Clearly, if $S$ is a bisaturated set in $\Lambda$ and $|S|>n$, then $\dim S >n-1$, as $V$ is in general linear position. So, $S=\Lambda \cap \aff S=\Lambda$.  Likewise, if $S$ is bisaturated in $\Lambda$ and $|S|<n$, then $|\Lambda\setminus S|>n$, and hence $S=\emptyset$.

On the other hand, if $|S|=n$, then $\dim (S)=n-1$ and $\dim (\Lambda\setminus S)=n-1$, as $V$ is in general linear position. This implies that 
\[|\Lambda\cap\aff S|\le \dim (\aff(\Lambda\cap \aff S))+1=\dim(S)+1=n=|S|,\]
and since $\Lambda\cap \aff S\supseteq S$ it follows that $\Lambda\cap\aff S=S$. Similarly $\Lambda\cap \aff(\Lambda\setminus S)=\Lambda\setminus S$. Thus, $\{S,\Lambda\setminus S\}$ is a bisaturated pair in $\Lambda$. It follows that $(\mathbb{R}^{n+1},K_V)$ has ${2n\choose n} +2$ bands. 
\end{example}

We conjecture that the optimal upper bound for the number of bands in an $(n+1)$-dimensional partially ordered vector space $(X,K)$ with closed generating cone $K$, is 
${2n\choose n}+2$. 

If the functionals $\{\phi_1,\ldots,\phi_m\}$ determining a polyhedral cone $K$ in $\mathbb{R}^n$ according to \eqref{f.polycone} are explicitly given, Theorem \ref{mainthm} can be used to list all bands in $(\R^n,K)$. Namely, after determining the linear dependence of $\phi_1,\ldots,\phi_m$, the saturation of every subset $M$ of $\{\phi_1,\ldots,\phi_m\}$ is known, and one can check whether $M$ is bisaturated. 
\begin{example}\label{ex.bands}
Let $X=\RR^4$, define 
\[
\phi_1=\left(\begin{smallmatrix}0\\2\\1\\1\end{smallmatrix}\right),\ 
\phi_2=\left(\begin{smallmatrix}2\\0\\1\\1\end{smallmatrix}\right),\ 
\phi_3=\left(\begin{smallmatrix}1\\1\\0\\2\end{smallmatrix}\right),\ 
\phi_4=\left(\begin{smallmatrix}1\\1\\2\\0\end{smallmatrix}\right),\ 
\phi_5=\left(\begin{smallmatrix}4\\0\\0\\0\end{smallmatrix}\right) 
\]
and consider $K=\{x\in\RR^4\colon \phi_i(x)\ge 0 \mbox{ for all } i\in\{1,\ldots,5\}\}$.
Since $\phi_1$, $\phi_2$, $\phi_3$, $\phi_5$ are linearly independent, 
$K$ 
is a cone in $\RR^4$.
For $u=\frac{1}{4}(1,1,1,1)^T$ we have $\phi_i(u)=1$ for all $i\in\{1,\ldots,5\}$, thus $u$ is an interior point of $K$, and we set  $\Sigma:=\left\{\phi\in K'\colon \phi(u)=1\right\}$. 
A straightforward calculation shows that $\phi_i$ is an extreme point of $\Sigma$ for all $i$, hence 
\[
\Lambda=\{\phi_1,\phi_2,\phi_3,\phi_4,\phi_5\}.
\]
Observe that $\spn\{\phi_1,\phi_2,\phi_3,\phi_4\}$ has dimension 3. Indeed, \[\phi_1+\phi_2-\phi_3-\phi_4=0,\] and $\phi_1,\phi_2,\phi_3$ are linearly independent.  We embed $(\RR^4,K)$ into $(\RR^5,\RR^5_+)$ under $\Phi$. We want to determine all bands in $(\mathbb{R}^4,K)$. Due to Proposition \ref{prop2}  
every band equals $Z(M)$ for some saturated set $M\subseteq \{\phi_1,\ldots,\phi_5\}$. Theorem \ref{mainthm} implies that for a saturated set $M\subseteq\{\phi_1,\ldots,\phi_5\}$ one has that $Z(M)$ is a band if and only if $M$ is bisaturated. We list the sets $N\subseteq\{1,\ldots,5\}$ for which $\{\phi_i\colon\, i\in N\}$ is bisaturated: $\emptyset$, $\{5\}$, $\{1,2\}$, $\{1,3\}$, $\{1,4\}$, $\{2,3\}$, $\{2,4\}$, $\{3,4\}$, $\{1,2,5\}$, $\{1,3,5\}$, $\{1,4,5\}$, $\{2,3,5\}$, $\{2,4,5\}$, $\{3,4,5\}$, $\{1,2,3,4\}$, $\{1,2,3,4,5\}$. 
So, we obtain 16 bands in $(\RR^4,K)$. 
\end{example}

Recall that a band $B$ in  a partially ordered vector space is {\em directed} if for each $x,y\in B$ there exists $z\in B$ such that $x\leq z$ and $y\leq z$. In the examples above the partially ordered vector spaces have  directed and non-directed bands, whereas in vector lattices all bands are directed.  The next example shows that there exists a partially ordered vector space whose non-trivial bands are all non-directed. Here it turns out that the carriers of $\widehat{B}$ in the extension procedures (I) and (II) (see Section 3) differ, in contrast to the situation in Example \ref{exa:namioka}. 
\begin{example}
Let $P$ be the polytope, with vertex set $v_1,\ldots,v_9$, depicted in Figure \ref{fig:1}. 

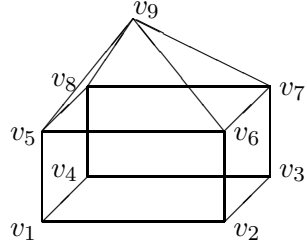
\begin{figure}[h]
\begin{center}
\setlength{\unitlength}{.6cm}
\begin{picture}(5,6)
  \linethickness{0.2mm}
  \put(0,0){\line(1,0){4}}
  \put(0,0){\line(0,1){2}}
  \put(4,0){\line(0,1){2}}
  \put(0,0){\line(1,1){1}}
  \put(1,1){\line(1,0){4}}
  \put(1,1){\line(0,1){2}}
  \put(1,3){\line(1,0){4}}
  \put(5,1){\line(0,1){2}}
  
  \put(0,2){\line(1,0){4}}
  \put(0,2){\line(1,1){1}}
  \put(4,2){\line(1,1){1}}
  \put(4,0){\line(1,1){1}}
  
  \put(0,2){\line(4,5){2}}
  \put(1,3){\line(2,3){1}}
  \put(4,2){\line(-4,5){2}}
  \put(5,3){\line(-2,1){3}}
  
  \put(-.7,-.3){$v_1$}
  \put(4.2,-.3){$v_2$}
  \put(5.2,1){$v_3$}
  \put(.2,1){$v_4$}
  \put(2,4.6){$v_9$}
  \put(-.7, 1.8){$v_5$}
  \put(4.2,1.8){$v_6$}
  \put(5.2,2.8){$v_7$}
  \put(.2,3){$v_8$}

  \thicklines
   \end{picture}
   \end{center}
   \caption{The polytope $P$}\label{fig:1}
   \end{figure}

The space $\mathbb{R}^4$ equipped with the polyhedral cone, \[K_P=\{x\in\mathbb{R}^4\colon \langle (v_i,1),x\rangle\geq 0\mbox{ for }i=1,\ldots, 9\},\] 
has two non-trivial bands corresponding to the bisaturated sets 
\[B_1=\{v_1,v_3,v_5,v_7,v_9\} \mbox{ and }  B_2=\{v_2,v_4,v_6,v_8,v_9\}.\]  
Both these bands are non-directed. 
Also observe that $\car(B_1)=\{v_2,v_4,v_6,v_8\}=\car(\Phi(B_1)^{dd})$ by Remark \ref{rem.indexed}, whereas $\car(\Phi(B_1^d)^d)=\{v_2,v_4,v_6,v_8,v_9\}$ by Theorem \ref{th:proc1}.
\end{example}

It is interesting to analyze the number of directed bands in an $n$-dimen\-sional  partially  ordered vector space $(X,K)$. In particular, it would be nice to know the answers to  following questions. Is the number of directed bands in an $n$-dimensional partially ordered vector space $(X,K)$, with generating closed cone $K$, bounded by $2^n$? If so, is the upper bound attained if and only if $(X,K)$ is an $n$-dimensional Archimedean vector lattice?

\begin{acknowledgement}
The authors thank Werner Timmermann for his helpful questions and comments.

Daniel Kr\"ahmann has given a first example of a space with a polyhedral cone with more bands than in the lattice case. The authors are grateful for that challenge. Marc Landgraf is acknowledged for contributing to the idea of Example \ref{nr_bands}. 

A.\ Kalauch thanks Ulrich Brehm for an inspiring discussion concerning a  geometrical characterization of bisaturated sets in spaces with polyhedral cones, which contributed to the observation in Lemma \ref{lemma.UB}.  
\end{acknowledgement}

 \vspace{10mm} \noindent
\begin{tabular}{l l l }
Anke Kalauch & Bas Lemmens & Onno van Gaans\\
FR Mathematik & School of Mathematics, &Mathematical Institute\\
Institut f\"ur Analysis & Statistics \& Actuarial Sciences &  Leiden University\\
TU Dresden & University of Kent & P.O. Box 9512\\
01062 Dresden& CT27NF, Canterbury & 2300 RA Leiden\\
Germany& United Kingdom & The Netherlands\\
 & B. Lemmens@kent.ac.uk & vangaans@math.leidenuniv.nl\\

\end{tabular}

\end{document}